\documentclass[11pt,twoside]{amsart}
\usepackage{amsmath, amsthm, amscd, amsfonts, amssymb, graphicx, color}
\usepackage[bookmarksnumbered, plainpages]{hyperref}
\def\b{\beta}
\def\bb{\mathbb}
\def\cal{\mathcal}
\def\rm{\textrm}
\def\bf{\textbf}

\def\F{\mathrm{Frat}}
\def\soc{\mathrm{soc}}
\def\aut{\mathrm{Aut}}

\DeclareMathOperator{\frat}{Frat}
\def\pn{\par\noindent}

\begin{document}


\title[Rational probabilistic zeta function]{A finiteness condition on the coefficients\\ of the probabilistic zeta function}
\author{Duong Hoang Dung and Andrea Lucchini}

\thanks{{\scriptsize
\hskip -0.4 true cm MSC(2010): Primary: 20E18; Secondary: {20D06, 20P05, 11M41.}
}}
\maketitle


\begin{abstract}  We discuss whether finiteness properties of a profinite group $G$ can be deduced from the coefficients of the probabilistic
zeta function $P_G(s)$. In particular we prove that if  $P_G(s)$  is rational and all but finitely many non abelian composition factors of $G$ are isomorphic to $PSL(2,p)$ for some prime $p$, then $G$ contains only finitely many maximal subgroups.
\end{abstract}

\newcommand\node[2]{\overset{#1}{\underset{#2}{\circ}}}
\newcommand{\DEF}[1]{{\em #1\/}}

\numberwithin{equation}{section}

\newtheorem{theorem}{\bf Theorem}[section]
\newtheorem{corollary}[theorem]{\bf Corollary}
\newtheorem{lemma}[theorem]{\bf Lemma}
\newtheorem{example}[theorem]{\bf Example}
\newtheorem{proposition}[theorem]{\bf Proposition}
\newtheorem{conjecture}[theorem]{\bf Conjecture}
\newtheorem{problem}[theorem]{\bf Problem}
\newtheorem{remark}[theorem]{\bf Remark}
\newtheorem{definition}[theorem]{\bf Definition}
\newtheorem*{Main1}{Theorem 1}
\newtheorem*{Main2}{Theorem 2}
\newtheorem*{Main3}{Corollary}
\newtheorem{claim}{\bf Claim}
\newtheorem*{my2}{Theorem \ref{my2}}
\newtheorem*{my3}{Theorem \ref{my3}}

\newcommand{\Prob}{\operatorname{Prob}}

\
\section{Introduction}

Let $G$ be a finitely generated profinite group. As $G$ has only
finitely many open  subgroups of a given index, for any $n \in
\mathbb N$ we may define the integer $a_n(G)$ as $a_n(G)=\sum_H
\mu_G(H),$ where the sum is over all open subgroups $H$ of $G$
with $|G:H|=n.$ Here $\mu_G(H)$ denotes the M\"obius function of
the poset of open subgroups of $G,$ which is defined by recursion
as follows: $\mu_G(G)=1$ and $\mu_G(H)=-\sum_{H<K}\mu_G(K)$ if
$H<G$.
 Then
we  associate to $G$ a formal Dirichlet series $P_G(s)$,  defined as
\[P_G(s)=\sum_{n \in \mathbb N}\frac{a_n(G)}{n^s}.\]
Notice that if $H$ is an open subgroup of $G$ and $\mu_G(H)\neq 0,$ then $H$ is an intersection of maximal subgroups of $G.$ Therefore the formal Dirichlet series $P_G(s)$ encodes information about
the lattice generated by the maximal subgroups of $G$, just as the Riemann zeta
function encodes information about the primes, and combinatorial properties of the probabilistic sequence $\{a_{n}(G)\}$ reflect
on the structure of $G.$

If $G$ contains only finitely many maximal subgroups (i.e. if the Frattini subgroup
$\frat G$ of $G$ has finite index in $G$), then there are only finitely many open subgroups $H$ of $G$ with $\mu_G(H)\neq 0$ and consequently $a_n(G)=0$ for all but finitely many $n\in \mathbb N$ (i.e. $P_G(s)$ is a finite Dirichlet series). A natural question is whether the converse is true.

Let $\{G_n\}_{n \in \mathbb N}$ be a countable
 descending series  of open normal
subgroups with the properties that $G_1=G,$ $\bigcap_{n \in
\mathbb N}G_n=1$ and $G_n/G_{n+1}$ is a chief factor of $G$ for
each $n \in \mathbb N.$ The factor group $G/G_n$ is finite, so the
Dirichlet  series $P_{G/G_n}(s)$ is also finite and belongs to the
ring $\cal D$ of Dirichlet polynomials with integer coefficients.
Actually,  $P_{G/G_{n}}(s)$ is a divisor of
$P_{G/G_{n+1}}(s)$ in the ring $\cal D$, i.e. there exists a
Dirichlet polynomial $P_n(s)$ such that
$P_{G/G_{n+1}}(s)=P_{G/G_{n}}(s)P_{n}(s)$. As explained in
\cite{DeLu06b}, the Dirichlet series $P_G(s)$ can be written
as an infinite formal product
\[P_G(s)=\prod_{n \in \mathbb
N}P_n(s),\]
 and if we change  the series $\{G_n\}_{n \in \mathbb
N}$, the factorization remains the same up to reordering the
factors. Moreover it turns out that $P_n(s)=1$ if
$G_n/G_{n+1}$ is a Frattini chief factor (i.e. $G_n/G_{n+1}\leq \frat(G/G_{n+1}$).
Notice that $G$ has finitely many open maximal subgroups if and only if the chief
series $\{G_n\}_{n \in \mathbb N}$ contains only finitely many non-Frattini factors.
This could suggest a wrong argument: if the product $P_G(s)=\prod_{n \in \mathbb
N}P_n(s)$ is a Dirichlet polynomial, then $P_n(s)=1$ for all but finitely many $n\in \mathbb N$
and consequently the series $\{G_n\}_{n \in \mathbb N}$ contains only finitely many non-Frattini factors.
The problem is that it is possible that a Dirichlet polynomial can be written as a
formal product of infinitely many non trivial elements of $\cal D$. To give an idea of what can occur,
let us recall a related question, with an unexpected solution: if $G$ is prosolvable, then
we can consider the $p$-local factor $$P_{G,p}(s)=\sum_{m\in \mathbb N}\frac{a_{p^m}}{p^{ms}}.$$
It turns out that $P_{G,p}(s)=\prod_{n \in \Omega_p}P_n(s)$ where $\Omega_p$ is the set
of indices  $n$ such that $G_n/G_{n+1}$ has $p$-power order. It is not difficult to prove that
a finitely generated prosolvable group $G$ contains only finitely many maximal subgroups whose index is a power of $p$
if and only if a chief series of $G$ contains only finitely many non-Frattini factors whose order is a $p$-power.
Therefore the previous tempting wrong argument would suggest the following conjecture: if the $p$-factor $P_{G,p}(s)$ is a
Dirichlet polynomial, then $G$ has only finitely many maximal subgroups of $p$-power index. However this is false; in \cite{DeLu06}
it is constructed a 2-generated  prosolvable group $G$ such that, for any prime $p$,
$G$ contains infinitely many maximal subgroups of $p$-power index, while $P_{G,p}(s)$ is a finite Dirichlet series.
Knowing that $P_{G,p}(s)$ can be a polynomial even when $\Omega_p$
is infinite, could lead to believe in the existence of a
counterexample to the conjecture that $P_G(s) \in \cal D$ implies
$G/\frat ( G)$ finite. However, using results from number theory,
in \cite{DeLu06} it was prove that if $G$ is prosolvable and $P_{G,p}(s)$ is a polynomial, then either $\Omega_p$
is finite or, for every prime $q$, there exists $n \in \Omega_p$
such that the dimension of $G_n/G_{n+1}$ as 
$\mathbb F_pG$-module is divisible by $q;$ using standard
arguments of modular representation theory one deduce that this is
possible only if infinitely many primes appears between the
divisors of the order of the finite images of $G$; but then
$P_{G,r}(s)\neq 1$ for infinitely many primes $r$ and $P_G(s)$
cannot be  a polynomial. So $P_G(s)$ can be a polynomial only if
$\Omega_p$ is finite for every prime and empty for all but
finitely many, and therefore only if $G/\frat (G)$ is finite.
Really a stronger result holds: if $G$ is a finitely generated prosolvable group, then $P_G(s)$ is rational (i.e. $P_G(s)=A(s)/B(s)$ with $A(s)$ and $B(s)$ finite Dirichlet series) if and only if $G/\frat G$ is a finite group.
Partial generalization has been obtained in \cite{DeLu07}
and in \cite{dung}. All these results can be summarized in the following statement:
\begin{theorem}Let $G$ be a finitely generated profinite group. Assume that there exist a prime $p$ and a normal open subgroup $N$ of $G$ such
that the set $\cal S$ of nonabelian composition factors of $N$ satisfies one of the following properties:
\begin{itemize}
\item all the groups in $\cal S$ are alternating groups;
\item all the groups in $\cal S$ are of Lie type over fields of characteristic $p,$
where $p$ is a fixed prime;
\item all the groups in $\cal S$ are sporadic simple groups.
\end{itemize}
Then $P_G(s)$ is rational if and only if $G/\F(G)$ is a finite group.
\end{theorem}

The main ingredient in the proof of the previous results is the following result, proved with the help of
the Skolem-Mahler-Lech Theorem and where $\pi(G)$ is the set of the primes $q$ with the properties that $G$ contains at least an open subgroup $H$ whose index is divisible by $q$.
\begin{proposition}\label{skolem remark}
Let $G$ be a finitely generated profinite group, assume that $\pi(G)$ is finite and let $r_i$ be the sequence of the composition lengths
of the non-Frattini factors in a chief series of $G$. Assume that there exists a positive integer $q$ and a sequence ${c_i}$ of nonnegative integers
such that the formal product
$$\prod_{i}\left(1-\frac{c_i}{(q^{r_i})^s}\right)$$
is rational. Then $c_i=0$ for all but finitely many indices $i.$
\end{proposition}

In our applications, ${c_i}/{(q^{r_i})^s}$ is one of the summands of the Dirichlet polynomial $P_i(s)$ associated to the chief factor
$G_i/G_{i+1}$ and must be chosen so that the polynomials $P_i^*(s)=1-{c_i}/{(q^{r_i})^s}$ satisfy two conditions:
\begin{enumerate}
\item if $P_i(s)\neq 1$ for infinitely many $i \in \mathbb N,$ then also $P_i^*(s)\neq 1$ for infinitely many $i \in \mathbb N;$
\item if the infinite product $\prod_{i}P_i(s)$ is rational, then also $\prod_{i}P_i^*(s)$ is rational.
\end{enumerate}

The choice of the \lq\lq approximation\rq\rq \ $P_i^*(s)$ of $P_i(s)$ is not easy and we are not able to use this strategy in the general case;
roughly speaking, it requires an order on the set of the primes numbers with the property that
\lq\lq small\rq\rq \ simple groups in $\cal S$
have order not divisible by \lq\lq large\rq\rq \ primes.
When $\cal S$ is the set of the alternating groups, we just consider the natural order, but when $\cal S$ consists
of simple groups of Lie type in characteristic $p$, we say that a prime $q_1$ is smaller than a prime $q_2$ if
the multiplicative order of $p$ mod $q_1$ is smaller than its order modulo $q_2.$ This order depends on the choice of $p$
 and our arguments do not work if $\cal S$ contains groups of different kinds, for example groups of Lie
type in different characteristics. However we expect that our statement remains true in the general case. In the present
paper we prove a new result in this direction:

\begin{theorem}\label{uno}
Let $G$ be a finitely generated profinite group. Assume that there exists a normal open subgroup $N$ of $G$ such
that any nonabelian composition factor of $N$ is isomorphic to $PSL(2,p)$ for some prime $p.$ Then $P_G(s)$ is rational if and only if $G/\F(G)$ is a finite group.
\end{theorem}

\section{Preliminaries and notations}\label{notation}
Let $\cal R$ be the ring of formal Dirichlet series with integer
coefficients. We say that
$F(s)=\sum _{n\in \bb{N}} {a_n}/{n^s}$ $\in \cal R$ is a Dirichlet polynomial if
$a_n=0$ for all but finitely many $n\in \bb{N}.$
The set $\cal D$ of the Dirichlet polynomials
is a subring of $\cal R.$ We will say that $F(s)\in \cal R$ is rational if there
exist $A(s), B(s)\in \cal D$ with $F(s)=A(s)/B(s).$

For every set $\pi$ of prime number, we consider the ring
endomorphism  of $\cal R$ defined by:
\begin{eqnarray*}
F(s)=\sum_{n\in \bb{N}} \frac{a_n}{n^s}  \mapsto  F^{\pi}(s)=\sum_{n\in \bb{N}} \frac{a_n^*}{n^s}
\end{eqnarray*}
where $a_n^*=0$ if $n$ is divisible by some prime $p\in \pi,$ $a_n^*=a_n$ otherwise.
We will use the following remark:
\begin{remark}\label{pirat}
For every set $\pi$ of prime numbers, if $F(s)$ is rational then $F^\pi(s)$ is rational.
\end{remark}

The following result is a consequence of the Skolem-Mahler-Lech Theorem (see \cite{DeLu06} for more details):
\begin{proposition}\label{skolem theorem}
Let $I\subseteq \bb{N}$ and let $q,r_i,c_i$ be positive integers for each $i\in I$. Assume that
\begin{itemize}
\item[(i)] for every $n\in\bb{N}$, the set $\{i\in I \mid r_i~\rm{divides}~n\}$ is finite;
\item[(ii)] there exists a prime $t$ such that $t$ does not divide $r_i$ for any $i\in I.$
\end{itemize}
If the product
$$F(s)=\prod_{i\in I}\left(1-\frac{c_i}{(q^{r_i})^s}\right)$$
is rational, then $I$ is finite.
\end{proposition}

\begin{proposition}{\cite[Corollary 5.2]{DeLu07}}\label{two conditions fulfilled}
Let $G$ be a finitely generated profinite group and assume that $\pi(G)$ is finite. For each $n$, there are only finitely many non-Frattini factors in a chief series whose composition length is at most $n$. Moreover there exists a prime $t$ such that no non-Frattini chief factor of $G$ has
composition length divisible by $t.$
\end{proposition}
\begin{proof}[Proof of Proposition \ref{skolem remark}] It follows immediately from Propositions \ref{skolem theorem} and \ref{two conditions fulfilled}.
\end{proof}
Finally let us recall the following result.
\begin{proposition}{\cite[Proposition 4.3]{DeLu07}}\label{prop 4.3}
Let $F(s)$ be a product of finite Dirichlet series:
$$F(s)=\prod_{i\in I}F_i(s),~\textrm{where}~ F_i(s)=\sum_{n\in\bb{N}}\frac{b_{i,n}}{n^s}$$
Let $q$ be a prime and $\Lambda$ the set of positive integers divisible by $q$. Assume that there exists a set $\{r_i\}_{i\in I}$  of positive integers such that if $n\in\Lambda$ and $b_{i,n}\ne 0$ then $n$ is an $r_i$-th power of some integer and $v_q(n)=r_i$ (where $v_q(n)$ is the $q$-adic valuation of $n$).  Define
$$w=\min\{x\in\bb{N}~\big|~ v_q(x)=1~\textrm{and}~b_{i,x^{r_i}}\ne 0~\textrm{for some}~i\in I\}$$
If $F(s)$ is rational, then the product
\begin{equation}
F^*(s)=\prod_{i\in I}\left(1+\frac{b_{i,w^{r_i}}}{(w^{r_i})^s}\right)
\end{equation}
is also rational.
\end{proposition}

Now let $G$ be a finitely generated profinite group and let $\{G_i\}_{i\in\bb{N}}$  be a fixed countable descending series of open normal subgroups with the property that $G_0=G$, $\bigcap_{i\in\bb{N}}G_i=1$ and $G_i/G_{i+1}$ is a chief factor of $G/G_{i+1}$ for each $i\in\bb{N}$. In particular, for
each $i \in \bb{N}$, there exist a simple group $S_i$ and a positive integer $r_i$ such that $G_i/G_{i+1}\cong S_i^{r_i}.$
Moreover, as described in \cite{DeLu06b}, for each $i\in \bb{N}$ a finite
Dirichlet series
\begin{equation}\label{eq-pi}
P_i(s)=\sum_{n\in \bb{N}}\frac{b_{i,n}}{n^s}
\end{equation}
is associated with the chief factor  $G_i/G_{i+1}$ and $P_G(s)$
can be written as an infinite formal product of the finite Dirichlet series $P_i(s)$:
\begin{equation}P_G(s)=\prod_{i\in \bb{N}}P_i(s).
\end{equation}
Moreover, this factorization is independent on the choice of chief series (see \cite{DeLu03,DeLu06b}) and $P_i(s)=1$ unless $G_i/G_{i+1}$ is a non-Frattini chief factor of $G$.

We recall some properties of the series $P_i(s).$ If $S_i$ is cyclic of order $p_i,$
then $P_i(s)=1-c_i/(p_i^{r_i})^s,$
where $c_i$ is the number of complements of $G_i/G_{i+1}$ in
$G/G_{i+1}.$ It is more difficult to compute the series $P_i(s) $
when $S_i$ is a non-abelian simple group. In that case a relevant
role is played by the group $L_i=G/C_G(G_i/G_{i+1}).$ This is a
monolithic primitive group and its unique minimal normal subgroup
is isomorphic to $G_i/G_{i+1}\cong S_i^{r_i}.$ If $n \neq
|S_i|^{r_i},$ then  the coefficient $b_{i,n}$ in (\ref{eq-pi}) depends only on the
knowledge of $L_i;$ more precisely we have
\begin{equation*}
b_{i,n}=\sum_{\substack{|L_i:H|=n\\L_i=H \soc(L_i)}}\mu_{L_i}(H).
\end{equation*}
Some help in computing the coefficients $b_{i,n}$ comes from the knowledge of the subgroup $X_i$ of $\aut
S_i$ induced by the conjugation action of the normalizer in $L_i$ of
a composition factor of the socle $S_i^r$ (note that $X_i$ is an
almost simple group with socle isomorphic to $S_i).$
More precisely, given an almost simple group
$X$ with socle $S$, we can consider the following Dirichlet polynomial:
\begin{equation}\label{pxs}P_{X,S}(s)=\sum_{n}\frac{c_n(X)}{n^s},~\textrm{\ where $c_n(X)=\sum_{\substack{|X:H|=n\\X=SN}}\mu_X(H)$.}
\end{equation}
The following can be deduced from \cite{Ser08}:
\begin{lemma}\label{paz}
If $S_i$ is nonabelian and $\pi$ is a set of primes containing at least one divisor of $|S_i|$ then
$$P_i^{\pi}(s)= P^{\pi}_{X_i,S_i}(r_is-r_i+1).$$
In particular, if $n$ is not divisible by some prime in $\pi$, then there exists $m\in \mathbb N$ with $n=m^{r_i}$ and $b_{i,n}=c_m(X_i)\cdot m^{r_i-1}.$
\end{lemma}

For an almost simple group $X$, let $\Omega(X)$ be the set of the odd integers
$m\in \mathbb N$ such that
\begin{itemize}
\item $X$ contains at least one subgroup $Y$ such that $X=Y\soc X$ and $|X:Y|=m;$
\item if $X=Y\soc X $ and $|X:Y|=m,$ then $Y$ is a maximal subgroup if $X.$
\end{itemize}
Note that if $m \in \Omega(X),$ $X=Y\soc X $ and $|X:Y|=m,$
then $\mu_X(Y)=-1$: in particular
$c_m(X)<0.$ Combined with Lemma \ref{paz}, this implies:

\begin{remark}\label{vice} If $m\in \Omega(X_i),$ then $b_{i,m^{r_i}}<0.$
\end{remark}

\begin{lemma}\label{minimalodd}
Let $X$ be an almost simple group with $\soc(X) =PSL(2,p)$, where $p\geq 5$ is an odd prime
and let $w(X)=\min\{m \mid m \in \Omega (X)\}.$ Then we have:
\begin{displaymath}
w(X)=\left\{ \begin{array}{ll} q(q-1)/2 & \textrm{if $q\equiv 3\!\!\mod 4$ and $q\notin \{5,7,11,19,29\},$}\\
q(q+1)/2 & \textrm{if $q\equiv 1\!\!\mod 4$ and $q\notin \{5,7,11,19,29\},$}\\
5 & \textrm{ if $q=5,$}\\
7 & \textrm{ if $q=7$ and $X=PSL(2,7),$}\\
3\cdot 7 & \textrm{ if $q=7$ and $X=PGL(2,7),$}\\
11 & \textrm{ if $q=11$ and $X=PSL(2,11),$}\\
11\cdot 5 & \textrm{ if $q=11$ and $X=PGL(2,11),$}\\
19\cdot 3 & \textrm{ if $q=19$ and $X=PSL(2,19),$}\\
19\cdot 3^2 & \textrm{ if $q=19$ and $X=PGL(2,19),$}\\
29\cdot 7& \textrm{ if $q=29$ and $X=PSL(2,29),$}\\
29\cdot 3\cdot 5& \textrm{ if $q=29$ and $X=PGL(2,29)$.}\\
\end{array}  \right.
\end{displaymath}
\end{lemma}

\begin{proof}Assume that $S\cong PSL(2,q)$ where $q\geq 5$ be an odd prime: $\aut(S)=PGL(2,q)$ and  $X=PSL(2,q)$ or $X=PGL(2,q)$. In both the cases the conclusion follows easily form the list of maximal subgroups of $X$ given in \cite{Dickson}.
\end{proof}

\section{Proof of Theorem\ref{uno}}
We start now the proof of our main result.
We assume that $G$ is a finitely generated profinite group $G$  with the properties that $P_G(s)=\sum_n a_n/n^s$ is rational.
As described in Section \ref{notation}, $P_G(s)$ can be written as a formal infinite product of Dirichlet polynomials $P_i(s)
=\sum_{n\in \bb{N}}{b_{i,n}}/{n^s}$
corresponding to the factors $G_i/G_{i+1}$ of a chief series of $G.$
Let $J$ be the set of indices $i$ such that $G_i/G_{i+1}$ is a non-Frattini chief factor.
Since $P_i(s)=1$ if $i\notin J,$ we have
$$P_G(s)=\prod_{j\in J}P_j(s).$$

For $C(s)=\sum_{n\in \mathbb N} c_n/n^s\in \cal R$, we define $\pi(C(s))$ to be the set of the primes $q$ for which there exists
at least one multiple $n$ of $q$ with $c_n\ne 0$. Notice that if $C(s)=A(s)/B(s)$ is rational then $\pi(C(s))\subseteq \pi(A(s))\bigcup \pi(B(s))$ is finite.
Let $\cal S$ be the set of the finite simple groups that are isomorphic to a composition factor of some non-Frattini chief factor of $G$. The first step in the proof of Theorem \ref{uno} is to show that $\cal S$ is finite. The proof of this claim requires the following result.
\begin{lemma}[{\cite[Lemma 3.1]{DeLu07}}]\label{number of non-iso composition factors}
Let $G$ be a finitely generated profinite group and let $q$ be a prime with $q\notin \pi(P_G(s))$.
If $q$ divides the order of a non-Frattini chief factor of $G,$ then this factor is not a $q$-group.
\end{lemma}

\begin{lemma}\label{gamma finite}If $G$ satisfies the hypothesis of Theorem \ref{uno}, then
the sets $\cal S$ and $\pi(G)$ are finite.
\end{lemma}
\begin{proof}
Since $P_G(s)$ is rational, we have that $\pi(P_G(s))$ is finite. Therefore, it follows from Lemma \ref{number of non-iso composition factors} that $\cal S$ contains only finitely many abelian groups. Assume by contradiction that $\cal S$ is infinite. This is possible only if the
subset $\cal S^*$ of  the simple groups in $\cal S$ that are isomorphic to $PSL(2,p)$ for some prime $p$ is infinite.
Let $$I:=\{j\in J\mid S_j \in \cal S^*\},
\ A(s):=\prod_{i\in I}P_i(s)\ \text{ and }\ B(s):=\prod_{i\notin I}P_i(s).$$ Notice that $\pi(B(s))\subseteq \bigcup_{S\in \cal S\setminus \cal S^*}\pi(S)$ is a finite set. Since $P_G(s)=A(s)B(s)$ and $\pi(P_G(s))$ is finite, if follows that the set $\pi(A(s))$ is finite. In particular,
there exists a prime number $q\geq 5$ such that $q\notin\pi(A(s))$ but $PSL(2,q)\in \cal S^*.$
Let $\Lambda$ be the set of the odd integers $n$ divisible by $q$ but not divisible by any prime strictly greater than $q$ and set
$$
\begin{aligned}
&r:=\min\{r_i \mid S_i=PSL(2,q)\},\\&I^*:=\{i\in I \mid S_i=PSL(2,q)\text { and }r_i=r\},\\
&w:=\min\{w(X_i)\mid S_i=PSL(2,q)\text{ and } r_i=r\},\\
&\beta:=\min\{n>1 \mid n \in \Lambda, v_q(n)=r \text{ and }b_{i,n} \neq 0 \text { for some } i \in I\}.
\end{aligned}$$
Assume $i\in I$, $n\in \Lambda$ and $b_{i,n}\neq 0.$
We have that $S_i\cong PSL(2,q_i)$ for a suitable prime $q_i$ and, by Lemma \ref{paz}, $n=x_i^{r_i}$ and $X_i$ contains a subgroup whose index divides $x_i$; using the classification of the
maximal subgroups of $X_i$ in \cite{Dickson}, one can notice that $q_i$ divides the indices of all
the subgroups of $X_i$ with odd index. Hence $q_i$ divides $x_i$ and consequently $n.$ Since $q_i$ is the largest prime divisor of $|S_i|$ and $q$ is the largest prime divisor of $n$, we deduce that $q=q_i.$
It follows that $\beta=w^r$ and $b_{i,\beta}\neq 0$ if and only if $i \in I^*$ and $w(X_i)=w;$ moreover in this last case
$b_{i,\beta} < 0.$
Hence the coefficient $c_{\b}$ of $1/{\b^s}$ in $A(s)$ is
$$c_\b=\sum_{i\in I^*,w(X_i)=w}b_{i,\b}<0.$$
This implies that $q\in\pi(A(s))$ which is a contradiction.
So we have proved that $\cal S$ is finite. By \cite[Lemma 3.2]{DeLu07}, if follows that $\pi(G)$ is also finite.
\end{proof}

\begin{proof}[{Proof of Theorem \ref{uno}}]
Let $\cal T$ be the set of the almost simple groups $X$ such that
there exists infinitely many $i\in J$ with $X_i\cong X$ 
and let $I=\{i\in J\mid X_i \in \cal T\}.$ By Lemma \ref{gamma finite}, $J\setminus I$ is finite.
We have to prove that $J$ is finite; this is equivalent to show that $I=\emptyset.$
Let $i \in I:$ by the hypothesis of Theorem \ref{uno}, there exists a prime $q_i$ such that $S_i\in \{C_{q_i},PSL(2,q_i)\}.$
Set $q = \max \{q_i \mid i\in I \}$
and let $\Lambda$ be the set of odd integers $n$ divisible by $q.$
Assume  $n\in\Lambda$ and $b_{i,n}\ne 0$ for some $i \in I.$
If $S_i$ is cyclic,
then $P_i(s)=1-c_n/n^s$ where $n=|G_i/G_{i+1}|=q_i^{r_i}$ and $c_n$ is the number of complements of $G_i/G_{i+1}$ in $G/G_{i+1}:$
this implies $q=q_i.$ Otherwise $S_i=PSL(2,q_i)$ with $q_i\leq q$ and, by Lemma \ref{paz}, $n=x_i^{r_i}$ and $X_i$ contains a subgroup whose index divides $x_i$; since $x_i$ divides $|SL(2,q_i)|$, $q$ divides $x_i$ and $q_i\leq q$, we must have $q=q_i.$
In both the cases, we have $n_i=x_i^{r_i}$ where $x_i$ a positive integer with
$v_q(x_i)=1.$
Let $$w=\min \{x\in\Lambda ~| ~v_q(x)=1 \text{ and } b_{i,x^{r_i}}\neq 0 \text{ for some }i \in I\}.$$
Since $J\setminus I$ is finite and $P_G(s)=\prod_{i\in J}P_i(s)$ is rational,
also $\prod_{i\in I}P_i(s)$ is rational.
In particular, the following series is rational:
$$Q(s)=\prod_{i\in I}P^{\{2\}}_i(s).$$
Let $I^*=\{i\in I\mid b_{i,w^{r_i}}\neq 0\}$. By the above considerations and Remark \ref{vice}, $i \in I^*$ if and only if
either $S_i\cong C_q$ and $w=q$ or $\soc(X_i)=PSL(2,q)$ and $w(X_i)=w$. In particular
if $i\in I^*$ then there exist infinitely many $j \in I$ with $X_i\cong X_j$ and all of them are in $I^*,$ 
hence $I^*$ is an infinite set. Moreover $b_{i,w^{r_i}}<0$ for every
$i \in I^*,$ and therefore
applying Proposition \ref{prop 4.3}to the Dirichlet series $Q(s),$ we deduce that the product
$$ H(s)=\prod_{i\in I}\left(1+\frac{b_{i,w^{r_i}}}{w^{r_is}}\right)=\prod_{i\in I^*}\left(1+\frac{b_{i,w^{r_i}}}{w^{r_is}}\right)$$
is  rational. By  Corollary \ref{skolem remark} the set $I^*$ must be finite, a contradiction.
\end{proof}


\bibliographystyle{alpha}

\bigskip

{\footnotesize \pn{\bf First Author}\;
\\ {Mathematisch Instituut, Leiden Universiteit, Niels Bohrweg 1, 2333 CA Leiden, The Netherlands}
\\ {Dipartimento di Matematica, Universit\`{a} degli studi di Padova,
Via Trieste 63, 35121 Padova, Italy\\
{\tt Email: dhdung1309@gmail.com}\\

{\footnotesize \pn{\bf Second Author}\; \\ {Dipartimento di Matematica, Universit\`{a} degli studi di Padova,
Via Trieste 63, 35121 Padova, Italy\\
{\tt Email: lucchini@math.unipd.it}\\
\end{document}